\documentclass[12pt,a4paper]{article}
\usepackage{a4wide}
\usepackage{amsmath,amssymb}
\usepackage[english]{babel}
\usepackage{graphicx}
\usepackage[T1]{fontenc}
\usepackage[numbers]{natbib}
\usepackage{theorem}
\usepackage{ifpdf}
\usepackage{hyperref}
\hypersetup{
    bookmarksopen=false,
    bookmarksnumbered=true,
    pdftitle={A Two-Queue Polling Model with Two Priority Levels in the First Queue},
    pdfauthor={M.A.A. Boon, I.J.B.F. Adan, O.J. Boxma}
}
\ifpdf
  \hypersetup{colorlinks=true,linkcolor=black,urlcolor=black,citecolor=black,pdfpagemode=UseOutlines,plainpages=false,pdfpagelabels}
\else
  \hypersetup{colorlinks=false}
\fi

\theorembodyfont{\upshape}
\theoremheaderfont{\bfseries}

\newtheorem{theorem}{Theorem}[section]
\newtheorem{property}[theorem]{Property}
\newtheorem{lemma}[theorem]{Lemma}
\newtheorem{remark}[theorem]{Remark}
\newtheorem{proposition}[theorem]{Proposition}
\newenvironment{proof}{\par\textbf{Proof:}\\}{\hfill$\square$\par}
\numberwithin{equation}{section}

\newcommand{\lz}{\lambda_1(1-z_1)+\lambda_2(1-z_2)}
\newcommand{\ee}{\textrm{e}}
\newcommand{\dd}{\,\textrm{d}}

\begin{document}
\title{A Two-Queue Polling Model with Two Priority Levels in the First Queue\footnote{The research was done in the framework of the BSIK/BRICKS project, and of the European Network of Excellence Euro-FGI.}\footnote{The present paper is an adapted and extended version of \cite{boonadanboxma2queues2008}.}}
\author{M.A.A. Boon\footnote{\textsc{Eurandom} and Department of Mathematics and Computer Science, Eindhoven University of Technology, P.O. Box 513, 5600MB Eindhoven, The Netherlands}\\\href{mailto:marko@win.tue.nl}{marko@win.tue.nl} \and I.J.B.F. Adan\footnotemark[2]\\\href{mailto:iadan@win.tue.nl}{iadan@win.tue.nl} \and O.J. Boxma\footnotemark[2]\\\href{mailto:boxma@win.tue.nl}{boxma@win.tue.nl}}

\date{May, 2008}

\maketitle

\begin{abstract}
In this paper we consider a single-server cyclic polling system consisting of two queues. Between visits to successive queues, the server is delayed by a random switch-over time. Two types of customers arrive at the first queue: high and low priority customers. For this situation the following service disciplines are considered: gated, globally gated, and exhaustive. We study the cycle time distribution, the waiting times for each customer type, the joint queue length distribution at polling epochs, and the steady-state marginal queue length distributions for each customer type.

\end{abstract}

\section{Introduction}\label{intro}

A polling model is a single-server system in which the server visits $n$ queues $Q_1, \dots, Q_n$ in cyclic order. Customers that arrive at $Q_i$ are referred to as type $i$ customers.
The special feature of the model considered in the present paper is that, within a customer type, we distinguish high and low priority customers. More specifically, we study a polling system which consists of two queues, $Q_1$ and $Q_2$. The first of these queues contains customers of two priority classes, high ($H$) and low ($L$). The exhaustive, gated and globally gated service disciplines are studied.

Our motivation to study a polling model with priorities is that the performance of a polling system can be improved through the introduction of priorities. In production environments, e.g., one could give highest priority to jobs with a service requirement below a certain threshold level. This might decrease the mean waiting time of an arbitrary customer without having to purchase additional resources \citep{wierman07}.
Priority polling models also can be used to study traffic intersections where conflicting traffic flows face a green light simultaneously; e.g. traffic which takes a left turn may have to give right of way to conflicting traffic that moves straight on, even if the traffic light is green for both traffic flows.
Another application is discussed in \cite{cicin2001}, where a priority polling model is used to study scheduling of surgery procedures in medical emergency rooms.
In the computer science community the Bluetooth and 802.11 protocols are frequently modelled as polling systems, cf. \cite{ieee802.11-1,bluetooth1,ieee802.11-2,bluetooth2}.
Many scheduling policies that have been considered or implemented in these protocols involve different priority levels in order to improve Quality-of-Service (QoS) for traffic that is very sensitive to delays or loss of data, such as Voice over Wireless IP.
The 802.11e amendment defines a set of QoS enhancements for wireless LAN applications by differentiating between high priority traffic, like streaming multimedia, and low priority traffic, like web browsing and email traffic.

Although there is quite an extensive amount of literature available on polling systems, only very few papers treat priorities in polling models. Most of these papers only provide approximations or focus on pseudo-conservation laws. In \cite{wierman07} exact mean waiting time results are obtained using the Mean Value Analysis (MVA) framework for polling systems, developed in \cite{winands06}. The MVA framework can only be used to find the first moment of the waiting time distribution for each customer type, and the mean residual cycle time. The main contribution of the present paper is the derivation of Laplace Stieltjes Transforms (LSTs) of the distributions of the marginal waiting times for each customer type; in particular it turns out to be possible to obtain exact expressions for the waiting time distributions of both high and low priority customers at a queue of a polling system. Probability Generating Functions (PGFs) are derived for the joint queue length distribution at polling epochs, and for the steady-state marginal queue length distribution of the number of customers at an arbitrary epoch. 

The present paper is structured as follows: Section \ref{general} gathers known results of nonpriority polling models which are relevant for the present study. Sections \ref{gated} (gated), \ref{globallygated} (globally gated), and \ref{exhaustive} (exhaustive) give new results on the priority polling model. In each of the sections we successively discuss the joint queue length distribution at polling epochs, the cycle time distribution, the marginal queue length distributions and waiting time distributions. The mean waiting times are given at the end of each section. A numerical example is presented in Section \ref{numericalexample} to illustrate some of the improvements that can be obtained by introducing prioritisation in a polling system.

\section{Notation and description of the nonpriority polling model}\label{general}

The model that is considered in this section, is a nonpriority polling model with two queues ($Q_1$ and $Q_2$).
We consider three service disciplines: gated, globally gated, and exhaustive. The gated service discipline states that during a visit to $Q_i$, the server serves only those type $i$ customers who are present at the polling epoch. All type $i$ customers that arrive during this visit will be served in the next cycle. In this respect, a cycle is the time between two successive visit beginnings to a queue. The exhaustive service discipline states that when the server arrives at $Q_i$, all type $i$ customers are served until no type $i$ customer is present in the system. We also consider the globally gated service discipline, which means that during a cycle only those customers will be served that were present at the beginning of that cycle.

Customers of type $i$ arrive at $Q_i$ according to a Poisson process with arrival rate $\lambda_i$ $(i = 1,2)$. Service times can follow any distribution, and we assume that a customer's service time is independent of other service times and independent of the arrival processes. The LST of the distribution of the generic service time $B_i$ of type $i$ customers is denoted by $\beta_i(\cdot)$.
The fraction of time that the server is serving customers of type $i$ equals $\rho_i := \lambda_i E(B_i)$. Switches of the server from $Q_i$ to $Q_{i+1}$ (all indices modulo 2), require a switch-over time $S_i$. The LST of this switch-over time distribution is denoted by $\sigma_i(\cdot)$. The fraction of time that the server is working (i.e., not switching) is $\rho := \rho_1+\rho_2$. We assume that $\rho < 1$, which is a necessary and sufficient condition for the steady state distributions of cycle times, queue lengths and waiting times to exist.

\cite{takacs68} studied this model, but without switch-over times and only with the exhaustive service discipline. \cite{coopermurray69} analysed this polling system for any number of queues, and for both gated and exhaustive service disciplines. \cite{eisenberg72} obtained results for a polling system with switch-over times (but only exhaustive service) by relating the PGFs of the joint queue length distributions at visit beginnings, visit endings, service beginnings and service endings. \cite{resing93} was the first to point out the relation between polling systems and Multitype Branching Processes with immigration in each state.
His results can be applied to polling models in which each queue satisfies the following property:

\begin{property}\label{resingproperty}
If the server arrives at $Q_i$ to find $k_i$ customers there, then during the course of the server's visit, each of these $k_i$ customers will effectively be replaced in an i.i.d. manner by a random population having probability generating function $h_i(z_1,\dots,z_n)$, which can be any $n$-dimensional probability generating function.
\end{property}

We use this property, and the relation to Multitype Branching Processes, to find results for our polling system with two queues, two priorities in the first queue, and gated, globally gated, and exhaustive service discipline. Notice that, unlike the gated and exhaustive service disciplines, the globally gated service discipline does not satisfy Property \ref{resingproperty}. But the results obtained by Resing also hold for a more general class of polling systems, namely those which satisfy the following (weaker) property that is formulated in \cite{semphd}:

\begin{property}\label{borstproperty}
If there are $k_i$ customers present at $Q_i$ at the beginning (or the end) of a visit to $Q_{\pi(i)}$, with $\pi(i) \in \{1, \dots, n\}$, then during the course of the visit to $Q_i$, each of these $k_i$ customers will effectively be replaced in an i.i.d. manner by a random population having probability generating function $h_i(z_1,\dots,z_n)$, which can be any $n$-dimensional probability generating function.
\end{property}

Globally gated and gated are special cases of the  synchronised gated service discipline, which states that only customers in $Q_i$ will be served that were present at the moment that the server reaches the ``parent queue'' of $Q_i$: $Q_{\pi(i)}$. For gated service, $\pi(i) = i$, for globally gated service, $\pi(i) = 1$. The synchronised gated service discipline is discussed in \cite{khamisy92}, but no observation is made that this discipline is a member of the class of polling systems satisfying Property \ref{borstproperty} which means that results as obtained in \cite{resing93} can be extended to this model.

\cite{borst97} combined the results of \cite{resing93} and \cite{eisenberg72} to find a relation between the PGFs of the marginal queue length distribution for polling systems with and without switch-over times, expressed in the Fuhrmann-Cooper queue length decomposition form \citep{fuhrmanncooper85}.

\subsection{Joint queue length distribution at polling epochs}

The probability generating function $h_i(z_1,\dots,z_n)$ which is mentioned in Property \ref{resingproperty} depends on the service discipline. In a polling system with two queues and gated service we have $h_i(z_1, z_2) = \beta_i(\lz)$. For exhaustive service this PGF becomes $h_i(z_1, z_2) = \pi_i(\sum_{j\neq i} \lambda_j(1-z_j))$, where $\pi_i(\cdot)$ is the LST of a busy period (BP) distribution in an $M/G/1$ system with only type $i$ customers, so it is the root of the equation $\pi_i(\omega) = \beta_i(\omega + \lambda_i(1 - \pi_i(\omega)))$. We choose the beginning of a visit to $Q_1$ as start of a cycle. In order to find the joint queue length distribution at the beginning of a cycle, we relate the numbers of customers in each queue at the beginning of a cycle to those at the beginning of the previous cycle. Customers always enter the system during a switch-over time, or during a visit period. The first group is called \emph{immigration}, whereas a customer from the second group is called \emph{offspring} of the customer that is served at the moment of his arrival. We define the immigration PGF for each switch-over time and the offspring PGF for each visit period analogous to \cite{resing93}. The immigration PGFs are:
\begin{align*}
g^{(2)}(z_1, z_2) &= \sigma_2(\lz), \\
g^{(1)}(z_1, z_2) &= \sigma_1(\lambda_1(1-z_1)+\lambda_2(1-h_2(z_1, z_2))).
\end{align*}
$g^{(2)}(z_1, z_2)$ is the PGF of the joint distribution of type $1$ and $2$ customers that arrive during $S_2$. For $S_1$ things are slightly more complicated, since type $2$ customers arriving during $S_1$ may be served before the end of the cycle, and generate offspring. $g^{(1)}(z_1, z_2)$ is the joint PGF of the type $1$ and $2$ customers present at the end of the cycle that either arrived during $S_1$, or are offspring of type 2 customers that arrived during $S_1$. The total immigration PGF is the product of these two PGFs:
\[
g(z_1, z_2) = \prod_{i=1}^2 g^{(i)}(z_1,z_2) = g^{(1)}(z_1,z_2)g^{(2)}(z_1,z_2).
\]
We define the offspring PGFs for each visit period in a similar manner:
\begin{align*}
f^{(2)}(z_1, z_2) &= h_2(z_1, z_2),\\
f^{(1)}(z_1, z_2) &= h_1(z_1, h_2(z_1, z_2)).
\end{align*}
The term for $Q_1$ is again slightly more complicated than the term for $Q_2$, since type 2 customers arriving during a server visit to $Q_1$ may be served before the end of the cycle, and generate offspring.

\cite{resing93} shows that the following recursive expression holds for the joint queue length PGF at the beginning of a cycle (starting with a visit to $Q_1$):
\begin{equation*}
P_1(z_1, z_2) = g(z_1, z_2)P_1\left(f^{(1)}(z_1, z_2), f^{(2)}(z_1, z_2)\right).\label{P1recursive}
\end{equation*}
This expression can be used to compute moments of the joint queue length distribution. Alternatively, iteration of this expression yields the following closed form expression for $P_1(z_1, z_2)$:
\begin{equation}
P_1(z_1, z_2) = \prod_{n=0}^\infty g(f_n(z_1, z_2)),\label{p1twoqueues}
\end{equation}
where we use the following recursive definition for $f_n(z_1, z_2)$, $n=0,1,2,\dots$:
\begin{align*}
f_n(z_1, z_2) &= (f^{(1)}(f_{n-1}(z_1, z_2)), f^{(2)}(f_{n-1}(z_1, z_2))), \\
f_0(z_1, z_2) &= (z_1, z_2).
\end{align*}
\cite{resing93} proves that this infinite product converges if and only if $\rho < 1$.

We can relate the joint queue length distribution at other polling epochs to $P_1(z_1, z_2)$. We denote the PGF of the joint queue length distribution at a visit beginning to $Q_i$ by $V_{b_i}(\cdot)$, so $P_1(\cdot) = V_{b_1}(\cdot)$. The PGF of the joint queue length distribution at a visit completion to $Q_i$ is denoted by $V_{c_i}(\cdot)$. The following relations hold:
\begin{align}
V_{b_1}(z_1, z_2) &= V_{c_2}(z_1, z_2)\sigma_2(\lz) \nonumber\\
&= V_{b_2}(z_1, h_2(z_1, z_2)) \sigma_2(\lz) \nonumber\\
&= V_{b_2}(z_1, f^{(2)}(z_1, z_2)) g^{(2)}(z_1, z_2), \label{vb1twoqueues}\\
V_{b_2}(z_1, z_2) &= V_{c_1}(z_1, z_2)\sigma_1(\lz) \nonumber\\
&= V_{b_1}(h_1(z_1, z_2), z_2) \sigma_1(\lz).\label{vb2twoqueues}
\end{align}

\subsection{Cycle time}

The cycle time, starting at a visit \emph{beginning} to $Q_1$, is the sum of the visit times to $Q_1$ and $Q_2$, and the two switch-over times which are independent of the visit times. Since type 2 customers who arrive during the visit to $Q_1$ or the switch from $Q_1$ to $Q_2$ will be served during the visit to $Q_2$, it can be shown that the LST of the distribution of the cycle time $C_1$, $\gamma_1(\cdot)$, is related to $P_1(\cdot)$ as follows:
\begin{equation}
\gamma_1(\omega) = \sigma_1(\omega + \lambda_2(1-\phi_2(\omega))) \, \sigma_2(\omega) \, P_1(\phi_1(\omega + \lambda_2(1-\phi_2(\omega))), \phi_2(\omega)),\label{cycletimelst}
\end{equation}
where $\phi_i(\cdot)$ is the LST of the distribution of the time that the server spends at $Q_i$ due to the presence of one type $i$ customer there. For gated service $\phi_i(\cdot) = \beta_i(\cdot)$, for exhaustive service $\phi_i(\cdot) = \pi_i(\cdot)$. A proof of \eqref{cycletimelst} can be found in \cite{boxmafralixbruin08}.

In some cases it is convenient to choose a different starting point for a cycle, for example when analysing a polling system with exhaustive service.
If we define $C_1^*$ to be the time between two successive visit \emph{completions} to $Q_1$, the LST of its distribution, $\gamma^*_1(\cdot)$, is:
\begin{align}
\gamma^*_1(\omega) =& \sigma_1(\omega + \lambda_1(1-\phi_1(\omega)) + \lambda_2(1-\phi_2(\omega + \lambda_1(1-\phi_1(\omega)))))\nonumber\\
&\cdot \sigma_2(\omega + \lambda_1(1-\phi_1(\omega))) \, V_{c_1}(\phi_1(\omega), \phi_2(\omega+\lambda_1(1-\phi_1(\omega)))),\label{cycletimlstVc}
\end{align}
with $V_{c_1}(z_1,z_2) = P_1(h_1(z_1, z_2), z_2)$.

\subsection{Marginal queue lengths and waiting times}

We denote the PGF of the steady-state marginal queue length distribution of $Q_1$ at the visit beginning by $\widetilde{V}_{b_1}(z) = V_{b_1}(z, 1)$. Analogously we define $\widetilde{V}_{b_2}(\cdot), \widetilde{V}_{c_1}(\cdot)$, and $\widetilde{V}_{c_2}(\cdot)$.
It is shown in \cite{borst97} that the steady-state marginal queue length of $Q_i$ can be decomposed into two parts: the queue length of the corresponding $M/G/1$ queue with only type $i$ customers, and the queue length at an arbitrary epoch during the intervisit period of $Q_i$, denoted by ${N_{i|I}}$. \cite{borst97} show that by virtue of PASTA, ${N_{i|I}}$ has the same distribution as the number of type $i$ customers seen by an arbitrary type $i$ customer arriving during an intervisit period, which equals
\[
E(z^{N_{i|I}}) = \frac{E(z^{N_{i|I_{\textit{begin}}}}) - E(z^{N_{i|I_{\textit{end}}}})}{(1-z)(E(N_{i|I_{\textit{end}}}) - E(N_{i|I_{\textit{begin}}}))},
\]
where $N_{i|I_{\textit{begin}}}$ is the number of type $i$ customers at the beginning of an intervisit period $I_i$, and $N_{i|I_{\textit{end}}}$ is the number of type $i$ customers at the end of $I_i$. Since the beginning of an intervisit period coincides with the completion of a visit to $Q_i$, and the end of an intervisit period coincides with the beginning of a visit, we know the PGFs for the distributions of these random variables: $\widetilde V_{c_i}(\cdot)$ and $\widetilde V_{b_i}(\cdot)$. This leads to the following expression for the PGF of the steady-state queue length distribution of $Q_i$ at an arbitrary epoch, $E[z^{N_i}]$:
\begin{equation}
E[z^{N_i}] = \frac{(1-\rho_i)(1-z)\beta_i(\lambda_i(1-z))}{\beta_i(\lambda_i(1-z))-z}
\cdot\frac{\widetilde{V}_{c_i}(z) - \widetilde{V}_{b_i}(z)}{(1-z)(E(N_{i|I_{\textit{end}}}) - E(N_{i|I_{\textit{begin}}}))}.   \label{queuelengthdecomposition}
\end{equation}
\cite{keilsonservi90} show that the distributional form of Little's law can be used to find the LST of the marginal waiting time distribution: $E(z^{N_i}) = E(\ee^{-\lambda_i(1-z)(W_i + B_i)})$, hence $E(\ee^{-\omega W_i}) = E[(1-\frac{\omega}{\lambda_i})^{N_i}]/\beta_i(\omega)$. This can be substituted into \eqref{queuelengthdecomposition}:
\begin{align}
E[\ee^{-\omega W_i}] =& \frac{(1-\rho_i)\omega}{\omega-\lambda_i(1-\beta_i(\omega))}\cdot
\frac{\widetilde{V}_{c_i}\left(1-\frac{\omega}{\lambda_i}\right) - \widetilde{V}_{b_i}\left(1-\frac{\omega}{\lambda_i}\right)}{(E(N_{i|I_{\textit{end}}}) - E(N_{i|I_{\textit{begin}}}))\omega/\lambda_i}   \nonumber\\
=& E[\ee^{-\omega W_{i|M/G/1}}]E\left[\left(1-\frac\omega{\lambda_i}\right)^{N_{i|I}}\right].\label{waitingtimedecomposition}
\end{align}
The interpretation of this formula is that the waiting time of a type $i$ customer in a polling model is the sum of two independent random variables: the waiting time of a customer in an $M/G/1$ queue with only type $i$ customers, $W_{i|M/G/1}$, and the remaining intervisit time for a customer that arrives at an arbitrary epoch during the intervisit time of $Q_i$.

For \emph{gated} service, the number of type $i$ customers at the beginning of a visit to $Q_i$ is exactly the number of type $i$ customers that  arrived during the previous cycle, starting at $Q_i$. In terms of PGFs: $\widetilde V_{b_i}(z) = \gamma_i(\lambda_i(1-z))$. The number of type $i$ customers at the end of a visit to $Q_i$ are exactly those type $i$ customers that arrived during this visit. In terms of PGFs: $\widetilde V_{c_i}(z) = \gamma_i(\lambda_i(1-\beta_i(\lambda_i(1-z))))$. We can rewrite $E(N_{i|I_{\textit{end}}}) - E(N_{i|I_{\textit{begin}}})$ as $\lambda_i E(I_i)$, because this is the number of type $i$ customers that arrive during an intervisit time. In Section \ref{momentsgeneral} we show that $\lambda_i E(I_i) = \lambda_i(1-\rho_i)E(C)$. Using these expressions we can rewrite Equation \eqref{waitingtimedecomposition} for gated service to:
\begin{equation}
E[\ee^{-\omega W_i}] = \frac{(1-\rho_i)\omega}{\omega-\lambda_i(1-\beta_i(\omega))}\cdot
\frac{\gamma_i(\lambda_i(1-\beta_i(\omega))) - \gamma_i(\omega)}{(1-\rho_i)\omega E(C)}.\label{lstwgated}
\end{equation}
For \emph{exhaustive} service, $\widetilde V_{c_i}(z) = 1$, because $Q_i$ is empty at the end of a visit to $Q_i$. The number of type $i$ customers at the beginning of a visit to $Q_i$ in an exhaustive polling system is equal to the number of type $i$ customers that arrived during the previous intervisit time of $Q_i$. Hence, $\widetilde{V}_{b_i}(z) = \widetilde{I}_i(\lambda_i(1-z))$, where $\widetilde{I}_i(\cdot)$ is the LST of the intervisit time distribution for $Q_i$. Substitution of $\widetilde{I}_i(\omega) = \widetilde{V}_{b_i}(1-\frac{\omega}{\lambda_i})$ in \eqref{waitingtimedecomposition} leads to the following expression for the LST of the steady-state waiting time distribution of a type $i$ customer in an exhaustive polling system:
\begin{equation}
E[\ee^{-\omega W_i}] = \frac{(1-\rho_i)\omega}{\omega-\lambda_i(1-\beta_i(\omega))}\cdot
\frac{1-\widetilde{I}_i(\omega)}{\omega E(I_i)}.\label{lstwexhaustive}
\end{equation}
To the best of our knowledge, the following result is new.
\begin{proposition}
Let the cycle time $C^*_i$ be the time between two successive visit \emph{completions} to $Q_i$. The LST of the cycle time distribution is given by \eqref{cycletimlstVc}. An equivalent expression for $E[\ee^{-\omega W_i}]$ if $Q_i$ is served exhaustively, is:
\begin{align}
E[\ee^{-\omega W_i}] &= \frac{1-\gamma^*_i(\omega - \lambda_i(1-\beta_i(\omega)))}{(\omega-\lambda_i(1-\beta_i(\omega)))E(C)}\label{lstwexhaustiveC}\\
&=  E[\ee^{-(\omega - \lambda_i(1-\beta_i(\omega))) C^*_{i,\textit{res}}}],\nonumber
\end{align}
where $C^*_{i,\textit{res}}$ is the residual length of $C^*_i$.
\end{proposition}
\begin{proof}
The cycle time is the length of an intervisit period $I_i$ plus the length of a visit $V_i$, which is the time required to serve all type $i$ customers that have arrived during $I_i$, and their type $i$ descendants. Hence, the following equation holds:
\begin{equation}
\gamma^*_i(\omega) = \widetilde{I}_i(\omega + \lambda_i(1-\pi_i(\omega))).\label{lstCexhaustiveI}
\end{equation}
We use this equation to find the inverse relation:
\begin{align*}
\widetilde{I}_i(\omega + \lambda_i(1-\pi_i(\omega))) &= \gamma^*_i(\omega) \\
&= \gamma^*_i(\omega + \lambda_i(1-\pi_i(\omega)) - \lambda_i(1-\pi_i(\omega)))\\
&= \gamma^*_i(\omega + \lambda_i(1-\pi_i(\omega)) - \lambda_i(1-\beta_i(\omega+\lambda_i(1-\pi_i(\omega))))).
\end{align*}
If we substitute $s := \omega + \lambda_i(1-\pi_i(\omega))$, we find
\begin{equation}
\widetilde{I}_i(s) = \gamma^*_i(s - \lambda_i(1-\beta_i(s))).\label{intervisitC}
\end{equation}
Substitution of \eqref{intervisitC} into \eqref{lstwexhaustive} gives \eqref{lstwexhaustiveC}.
\end{proof}
\begin{remark}
We can write \eqref{lstCexhaustiveI} and \eqref{intervisitC} as follows:
\[
\gamma^*_i(\omega) = \widetilde{I}_i(\psi(\omega)), \qquad
\widetilde{I}_i(s) = \gamma^*_i(\phi(s)),
\]
where $\phi(\cdot)$ equals the Laplace exponent of the L\'evy process $\sum_{j=1}^{N(t)}B_{i,j}-t$, with $N(t)$ a Poisson process with intensity $\lambda_i$, and with $\psi(\omega) = \omega+\lambda_i(1-\pi_i(\omega))$, which is known to be the inverse of $\phi(\cdot)$.
\end{remark}

\subsection{Moments}\label{momentsgeneral}

The focus of this paper is on LST and PGF of distribution functions, not on their moments. Moments can be obtained by differentiation, and are also discussed in \cite{wierman07}. In this subsection we will only mention some results that will be used later.

First we will derive the mean cycle time $E(C)$. Unlike higher moments of the cycle time, the mean does not depend on where the cycle starts: $E(C) = \frac{E(S_1) + E(S_2)}{1-\rho}$. This can easily be seen, because $1-\rho$ is the fraction of time that the server is not working, but switching. The total switch-over time is $E(S_1) + E(S_2)$.

The expected length of a visit to $Q_i$ is $E(V_i) = \rho_i E(C)$. The mean length of an intervisit period for $Q_i$ is $E(I_i) = (1-\rho_i)E(C)$. Notice that these expectations do not depend on the service discipline used. The expected number of type $i$ customers at polling moments does depend on the service discipline. For gated service the expected number of type $i$ customers at the beginning of a visit to $Q_i$ is $\lambda_i E(C)$. For exhaustive service this is $\lambda_i E(I_i)$. The expected number of type $i$ customers at the beginning of a visit to $Q_{i+1}$ is $\lambda_i (E(V_i) + E(S_i))$ for gated service, and $\lambda_i E(S_i)$ for exhaustive service.

Moments of the waiting time distribution for a type $i$ customer at an arbitrary epoch can be derived from the LSTs given by \eqref{lstwgated}, \eqref{lstwexhaustive} and \eqref{lstwexhaustiveC}. We only present the first moment:
\begin{align}
&\textrm{Gated: } & E(W_i) &= (1+\rho_i)\frac{E(C_i^2)}{2E(C)},\label{ewgated}\\
&\textrm{Exhaustive: } & E(W_i) &= \frac{E(I_i^2)}{2E(I_i)}+\frac{\rho_i}{1-\rho_i}\frac{E(B_i^2)}{2E(B_i)},\nonumber\\
& &  &= (1-\rho_i)\frac{E({C^*_i}^2)}{2E(C)}.\label{ewexhaustiveC}
\end{align}
Notice that the start of $C_i$ is the \emph{beginning} of a visit to $Q_i$, whereas the start of $C^*_i$ is  the \emph{end} of a visit. Equations \eqref{ewgated} and \eqref{ewexhaustiveC} are in agreement with Equations (4.1) and (4.2) in \cite{boxmaworkloadsandwaitingtimes89}. Although at first sight these might seem nice, closed formulas, it should be noted that the expected residual cycle time and the expected residual intervisit time are not easy to determine, requiring the solution of a large set of equations. MVA is an efficient technique to compute mean waiting times, the mean residual cycle time, and also the mean residual intervisit time. We refer to \cite{winands06} for an MVA framework for polling models.

\section{Gated service}\label{gated}

In this section we study the gated service discipline for a polling system with two queues and two priority classes in the first queue: high ($H$) and low ($L$) priority customers. All type $H$ and $L$ customers that are present at the moment when the server arrives at $Q_1$, will be served during the server's visit to $Q_1$. First all type $H$ customers will be served, then all type $L$ customers. Type $H$ customers arrive at $Q_1$ according to a Poisson process with intensity $\lambda_H$, and have a service requirement $B_H$ with LST $\beta_H(\cdot)$. Type $L$ customers arrive at $Q_1$ with intensity $\lambda_L$, and have a service requirement $B_L$ with LST $\beta_L(\cdot)$. If we do not distinguish between high and low priority customers, we can still use the results from Section \ref{general} if we regard the system as a polling system with two queues where customers in $Q_1$ arrive according to a Poisson process with intensity $\lambda_1 := \lambda_H + \lambda_L$ and have service requirement $B_1$ with LST $\beta_1(\cdot) = \frac{\lambda_H}{\lambda_1} \beta_H(\cdot) + \frac{\lambda_L}{\lambda_1} \beta_L(\cdot)$.

We follow the same approach as in Section \ref{general}. First we study the joint queue length distribution at polling epochs, then the cycle time distribution, followed by the marginal queue length distribution and waiting time distribution. The last subsection provides the first moment of these distributions.

\subsection{Joint queue length distribution at polling epochs}

Equations \eqref{vb1twoqueues} and \eqref{vb2twoqueues} give the PGFs of the joint queue length distribution at visit beginnings, $V_{b_i}(z_1, z_2)$. A type 1 customer entering the system is a type $H$ customer with probability $\lambda_H/\lambda_1$, and a type $L$ customer with probability $\lambda_L/\lambda_1$. We can express the PGF of the joint queue length distribution in the polling system with priorities, $V_{b_i}(\cdot,\cdot,\cdot)$, in terms of the PGF of the joint queue length distribution in the polling system without priorities, $V_{b_i}(\cdot,\cdot)$.

\begin{lemma}
\label{p1_3qs}
\begin{equation}
V_{b_i}(z_H, z_L, z_2) = V_{b_i}\left(\frac{\lambda_H z_H+\lambda_L z_L}{\lambda_1}, z_2\right).
\end{equation}
\end{lemma}
\begin{proof}
Let $X_H$ be the number of high priority customers present in $Q_1$ at the beginning of a visit to $Q_i$, $i=1,2$. Similarly define $X_L$ to be the number of low priority customers present in $Q_1$ at the beginning of a visit to $Q_i$. Let $X_1 = X_H + X_L$. Since the type $H$/$L$ customers in $Q_1$ are exactly those $H$/$L$ customers that arrived since the previous visit beginning at $Q_i$, we know that
\[P(X_H=i,X_L=k-i|X_1=k) = \binom ki \left(\frac{\lambda_H}{\lambda_1}\right)^i \left(\frac{\lambda_L}{\lambda_1}\right)^{k-i}.\]
Hence
\begin{align*}
E[z_H^{X_H}z_L^{X_L}|X_1 = k]
&= \sum_{i=0}^\infty\sum_{j=0}^\infty z_H^iz_L^j P(X_H=i,X_L=j|X_1=k) \\
&=\left(\frac{\lambda_H z_H + \lambda_L z_L}{\lambda_1}\right)^k.
\end{align*}
Finally,
\begin{align*}
V_{b_i}(z_H, z_L, z_2)
 &= \sum_{i=0}^\infty\sum_{j=0}^\infty \left(\frac{\lambda_H z_H + \lambda_L z_L}{\lambda_1}\right)^i z_2^j P(X_1=i, X_2=j) \\
&= V_{b_i}\left(\frac{1}{\lambda_1}(\lambda_H z_H+\lambda_L z_L), z_2\right).
\end{align*}
\hfill\mbox{}  
\end{proof}

\subsection{Cycle time}

The LST of the cycle time distribution is still given by \eqref{cycletimelst} if we define $\lambda_1 := \lambda_H + \lambda_L$ and $\beta_1(\cdot) := \frac{\lambda_H}{\lambda_1} \beta_H(\cdot) + \frac{\lambda_L}{\lambda_1} \beta_L(\cdot)$, because the cycle time does not depend on the order of service.

Equation \eqref{cycletimelst} is valid for polling systems with queues having any branching type service discipline. In the present section we can derive an alternative, shorter expression for $\gamma_1(\cdot)$ by explicitly using the fact that $Q_1$ receives gated service. The type 1 (i.e. both $H$ and $L$) customers present at the visit beginning to $Q_1$ are those that arrived during the previous cycle: $P_1(z, 1) = \gamma_1(\lambda_1(1-z))$. By setting $\omega = \lambda_1(1-z)$, this leads to the following expression for the LST of the distribution of $C_1$ if service in $Q_1$ is gated:
\begin{equation}
\gamma_1(\omega) = P_1(1-\frac{\omega}{\lambda_1}, 1).\label{lstCgatedShort}
\end{equation}

\subsection{Marginal queue lengths and waiting times}\label{gatedmarginalw}

We first determine the LST of the waiting time distribution for a type $L$ customer, using the fact that this customer will not be served until the next cycle (starting at $Q_1$). The time from the start of the cycle until the arrival will be called ``past cycle time'', denoted by $C_{1P}$. The residual cycle time will be denoted by $C_{1R}$. The waiting time of a type $L$ customer is composed of $C_{1R}$, the service times of all high priority customers that arrived during $C_{1P}+C_{1R}$, and the service times of all low priority customers that have arrived during $C_{1P}$. Let $N_H(T)$ be the number of high priority customers that have arrived during time interval $T$, and equivalently define $N_L(T)$.

\begin{theorem}
\begin{align*}
E\left[\ee^{-\omega W_L}\right] =& \frac{\gamma_1(\lambda_H(1-\beta_H(\omega))+\lambda_L(1-\beta_L(\omega)))
- \gamma_1(\omega+\lambda_H(1-\beta_H(\omega)))}{(\omega-\lambda_L(1-\beta_L(\omega)))E(C)}.
\end{align*}
\end{theorem}

\noindent\begin{proof}
\begin{align}
E\left[\ee^{-\omega W_L}\right]
&= E\left[\ee^{-\omega (C_{1R}+\sum_{i=1}^{N_H(C_{1P}+C_{1R})}B_{H,i}+\sum_{i=1}^{N_L(C_{1P})}B_{L,i})}\right] \nonumber\\
&= \int_{t=0}^\infty \int_{u=0}^\infty \sum_{m=0}^\infty\sum_{n=0}^\infty E\left[\ee^{-\omega\sum_{i=1}^{m}B_{H,i}}\right]E\left[\ee^{-\omega\sum_{i=1}^{n}B_{L,i}}\right]
\nonumber\\
&\quad\cdot\ee^{-\omega u} \frac{(\lambda_H(t+u))^m}{m!}\ee^{-\lambda_H(t+u)}\frac{(\lambda_L t)^n}{n!}\ee^{-\lambda_Lt} \dd P(C_{1P}<t, C_{1R}<u)\nonumber\\
\nonumber\\
&= \int_{t=0}^\infty \int_{u=0}^\infty \ee^{-t(\lambda_H(1-\beta_H(\omega))+\lambda_L(1-\beta_L(\omega)))}\ee^{-u(\omega + \lambda_H(1-\beta_H(\omega)))} \dd P(C_{1P}<t, C_{1R}<u)\nonumber\\
&=\frac{\gamma_1(\lambda_H(1-\beta_H(\omega))+\lambda_L(1-\beta_L(\omega))) - \gamma_1(\omega+\lambda_H(1-\beta_H(\omega)))}{(\omega-\lambda_L(1-\beta_L(\omega)))E(C)}.
\label{lstwlgated}
\end{align}
For the last step in the derivation of \eqref{lstwlgated} we used
\[E[\ee^{-\omega_P C_{1P}-\omega_R C_{1R}}] = \frac{E[\ee^{-\omega_P C_1}]-E[\ee^{-\omega_R C_1}]}{(\omega_R-\omega_P)E(C)},\]
which is obtained in \cite{boxmalevyyechiali92}.
\end{proof}

\begin{remark}
The Fuhrmann-Cooper decomposition \citep{fuhrmanncooper85} still holds for the waiting time of type $L$ customers, because \eqref{lstwlgated} can be rewritten into
\begin{align}
E\left[\ee^{-\omega W_L}\right] =& \frac{(1-\rho_L)\omega}{\omega-\lambda_L(1-\beta_L(\omega))} \nonumber\\
&\cdot
\frac{\gamma_1(\lambda_H(1-\beta_H(\omega))+\lambda_L(1-\beta_L(\omega))) - \gamma_1(\omega+\lambda_H(1-\beta_H(\omega)))}{(1-\rho_L)\omega E(C)}.\label{lstwlgateddecomposition}
\end{align}
We recognise the first term on the right-hand side of \eqref{lstwlgateddecomposition} as the LST of the waiting time distribution of an $M/G/1$ queue with only type $L$ customers. An interpretation of the other two terms on the right-hand side can be found when regarding the polling system as a polling system with \emph{three} queues $(Q_H, Q_L, Q_2)$ and no switch-over time between $Q_H$ and $Q_L$. The service discipline of this equivalent system is synchronised gated, which is a more general version of gated. The gates for queues $Q_H$ and $Q_L$ are set simultaneously when the server arrives at $Q_H$, but the gate for $Q_2$ is still set when the server arrives at $Q_2$. In the following paragraphs we show that the second and third term on the right-hand side of \eqref{lstwlgateddecomposition} together can be interpreted as $E[\left(1-\frac\omega{\lambda_L}\right)^{N_{L|I}}]$, where $N_{L|I}$ is the number of type $L$ customers at a random epoch during the intervisit period of $Q_L$.
\end{remark}

The expression for the LST of the distribution of the number of type $L$ customers at an arbitrary epoch is determined by first converting the waiting time LST to sojourn time LST, i.e., multiplying expression \eqref{lstwlgateddecomposition} with $\beta_L(\omega)$. Second, we apply the distributional form of Little's law \citep{keilsonservi90} to \eqref{lstwlgateddecomposition}. This law can be applied because the required conditions are fulfilled for each customer class (H, L, and 2): the customers enter the system in a Poisson stream, every customer enters the system and leaves the system one at a time in order of arrival, and for any time $t$ the entry process into the system of customers after time $t$ and the time spent in the system by any customer arriving before time $t$ are independent. The result is:
\begin{equation}
E\left[z^{N_L}\right] = \frac{(1-\rho_L)(1-z)\beta_L(\lambda_L(1-z))}{\beta_L(\lambda_L(1-z))-z}
\cdot \frac{\widetilde{V}_{c_L}(z) - \widetilde{V}_{b_L}(z)}{(1-z)(E(N_{L|I_{\textit{end}}}) - E(N_{L|I_{\textit{begin}}}))}.\label{gfzlgateddecomposition}
\end{equation}
In this equation $\widetilde{V}_{b_L}(z)$ denotes the PGF of the distribution of the number of type $L$ customers at the beginning of a visit to $Q_L$, and $\widetilde{V}_{c_L}(z)$ denotes the PGF at the completion of a visit to $Q_L$:
\begin{align*}
\widetilde{V}_{b_L}(z) &= V_{b_1}(\beta_H(\lambda_L(1-z)), z, 1)\\
&= \gamma_1(\lambda_H(1-\beta_H(\lambda_L(1-z))) + \lambda_L(1-z)), \\
\widetilde{V}_{c_L}(z) &= V_{b_1}(\beta_H(\lambda_L(1-z)), \beta_L(\lambda_L(1-z)), 1)\\
&= \gamma_1(\lambda_H(1-\beta_H(\lambda_L(1-z))) + \lambda_L(1-\beta_L(\lambda_L(1-z)))).
\end{align*}
The last term in \eqref{gfzlgateddecomposition} is the PGF of the distribution of the number of type $L$ customers at an arbitrary epoch during the intervisit period of $Q_L$, $E[z^{N_{L|I}}]$. Substitution of $\omega := \lambda_L(1-z)$ in \eqref{gfzlgateddecomposition}, and using $(E(N_{L|I_{\textit{end}}}) - E(N_{L|I_{\textit{begin}}})) = \lambda_L E(I_L)$, shows that the second and third term at the right-hand side of \eqref{lstwlgateddecomposition} together indeed equal $E[\left(1-\frac\omega{\lambda_L}\right)^{N_{L|I}}]$.

The derivation of the LSTs of $W_H$ and $W_2$ is similar and leads to the following expressions:
\begin{align}
E\left[\ee^{-\omega W_H}\right] =& \frac{(1-\rho_H)\omega}{\omega-\lambda_H(1-\beta_H(\omega))} \cdot
\frac{\gamma_1(\lambda_H(1-\beta_H(\omega))) - \gamma_1(\omega)}{(1-\rho_H)\omega E(C)},\label{lstwhgated}\\
E\left[\ee^{-\omega W_2}\right] =& \frac{(1-\rho_2)\omega}{\omega-\lambda_2(1-\beta_2(\omega))} \cdot
\frac{\gamma_2(\lambda_2(1-\beta_2(\omega))) - \gamma_2(\omega)}{(1-\rho_2)\omega E(C)}.\label{lstw2gated}
\end{align}
\begin{remark}
Equations \eqref{lstwhgated} and \eqref{lstw2gated} are equivalent to the LST of $W_i$ in a nonpriority polling system \eqref{lstwgated}, which illustrates that the Fuhrmann-Cooper decomposition also holds for the waiting time distributions of high priority customers in $Q_1$ and type 2 customers in a polling system with gated service.
\end{remark}

Application of the distributional form of Little's law to these expressions results in:
\begin{align*}
E\left[z^{N_H}\right] &= \frac{(1-\rho_H)(1-z)\beta_H(\lambda_H(1-z))}{\beta_H(\lambda_H(1-z))-z} \cdot
\frac{\gamma_1(\lambda_H(1-\beta_H(\lambda_H(1-z)))) - \gamma_1(\lambda_H(1-z))}{\lambda_H(1-\rho_H)(1-z)E(C)},\\
E\left[z^{N_2}\right] &= \frac{(1-\rho_2)(1-z)\beta_2(\lambda_2(1-z))}{\beta_2(\lambda_2(1-z))-z} \cdot
\frac{\gamma_2(\lambda_2(1-\beta_2(\lambda_2(1-z)))) - \gamma_2(\lambda_2(1-z))}{\lambda_2(1-\rho_2)(1-z)E(C)}.
\end{align*}
\begin{remark}
If the service discipline in $Q_2$ is not gated, but another branching type service discipline that satisfies Property \ref{resingproperty}, \eqref{lstw2gated} should be replaced by the more general expression \eqref{waitingtimedecomposition}.
\end{remark}

\subsection{Moments}\label{momentsgated}

As mentioned in Section \ref{momentsgeneral}, we do not focus on moments in this paper, and we only mention the mean waiting times of type $H$ and $L$ customers. For a type $H$ customer, it is immediately clear that $E(W_H) = (1+\rho_H)E(C_{1,\textit{res}})$. The mean waiting time for a type $L$ customer can be obtained by differentiating \eqref{lstwlgated}. This results in:
\[E(W_L) = (1+2\rho_H+\rho_L)E(C_{1,\textit{res}}).\]
These formulas can also be obtained using MVA, as shown in \cite{wierman07}.

\section{Globally gated service}\label{globallygated}

In this section we discuss a polling model with two queues $(Q_1, Q_2)$ and two priority classes ($H$ and $L$) in $Q_1$ with globally gated service. For this service discipline, only customers that were present when the server started its visit to $Q_1$ are served. This feature makes the model exactly the same as a nonpriority polling model with three queues $(Q_H, Q_L, Q_2)$. Although this system does not satisfy Property \ref{resingproperty}, it does satisfy Property \ref{borstproperty} which implies that we can still follow the same approach as in the previous sections.

\subsection{Joint queue length distribution at polling epochs}

We define the beginning of a visit to $Q_1$ as the start of a cycle, since this is the moment that determines which customers will be served during the next visits to the queues. Arriving customers will always be served in the next cycle, so the three $(i=H,L,2)$ offspring PGFs are:
\begin{align*}
f^{(i)}(z_H, z_L, z_2) &= h_i(z_H, z_L, z_2) \\
&= \beta_i(\lambda_H(1-z_H)+\lambda_L(1-z_L)+\lambda_2(1-z_2)),
\end{align*}
The two $(i=1,2)$ immigration functions are:
\[
g^{(i)}(z_H, z_L, z_2) = \sigma_i(\lambda_H(1-z_H)+\lambda_L(1-z_L)+\lambda_2(1-z_2)),
\]
Using these definitions, the formula for the PGF of the joint queue length distribution at the beginning of a cycle is similar to the one found in Section \ref{general}:
\begin{equation}
P_1(z_H, z_L, z_2) = \prod_{n=0}^\infty g(f_n(z_H, z_L, z_2)).
\end{equation}

Notice that in a system with globally gated service it is possible to express the joint queue length distribution at the beginning of a cycle in terms of the cycle time LST, since all customers that are present at the beginning of a cycle are exactly all of the customers that have arrived during the previous cycle:

\begin{equation}
P_1(z_H, z_L, z_2) = \gamma_1(\lambda_H(1-z_H)+\lambda_L(1-z_L)+\lambda_2(1-z_2)).\label{p1globallygated}
\end{equation}

\subsection{Cycle time}

Since only those customers that are present at the start of a cycle, starting at $Q_1$, will be served during this cycle, the LST of the cycle time distribution is
\begin{equation}
\gamma_1(\omega) = \sigma_1(\omega)\sigma_2(\omega)P_1(\beta_H(\omega), \beta_L(\omega), \beta_2(\omega)).\label{lstcgloballygated}
\end{equation}
Substitution of \eqref{p1globallygated} into this expression gives us the following relation:
\begin{align*}
&\gamma_1(\omega) = \sigma_1(\omega)\sigma_2(\omega)\\
&\cdot\gamma_1(\lambda_H(1-\beta_H(\omega))+\lambda_L(1-\beta_L(\omega))+\lambda_2(1-\beta_2(\omega))).
\end{align*}
\cite{boxmalevyyechiali92} show that this relation leads to the following expression for the cycle time LST:
\[
\gamma_1(\omega) = \prod_{i=0}^\infty \sigma(\delta^{(i)}(\omega)),
\]
where $\sigma(\cdot) = \sigma_1(\cdot)\sigma_2(\cdot)$, and $\delta^{(i)}(\omega)$ is recursively defined as follows:
\begin{align*}
\delta^{(0)}(\omega) &= \omega,\\
\delta^{(i)}(\omega) &= \delta(\delta^{(i-1)}(\omega)), \qquad\qquad i=1,2,3,\dots,\\
\delta(\omega) &= \lambda_H(1-\beta_H(\omega)) + \lambda_L(1-\beta_L(\omega)) + \lambda_2(1-\beta_2(\omega)).
\end{align*}

\subsection{Marginal queue lengths and waiting times}

For type $H$ and $L$ customers, the expressions for $E(\ee^{-\omega W_H})$ and $E(\ee^{-\omega W_L})$ are exactly the same as the ones found in Section \ref{gatedmarginalw}, but with $\gamma_1(\cdot)$ as defined in \eqref{lstcgloballygated}.

The expression for $E(\ee^{-\omega W_2})$ can be obtained with the method used in Section \ref{gatedmarginalw}:
\begin{align*}
E\left[\ee^{-\omega W_2}\right] 
=&\sigma_1(\omega)\cdot\frac{\gamma_1(\sum_{i=H,L,2}\lambda_i(1-\beta_i(\omega))) - \gamma_1(\omega+\sum_{i=H,L}\lambda_i(1-\beta_i(\omega)))}{(\omega-\lambda_2(1-\beta_2(\omega)))E(C)}\\
=& \sigma_1(\omega)\cdot\frac{(1-\rho_2)\omega}{\omega-\lambda_2(1-\beta_2(\omega))} \\
&\cdot
\frac{\gamma_1(\sum_{i=H,L,2}\lambda_i(1-\beta_i(\omega))) - \gamma_1(\omega+\sum_{i=H,L}\lambda_i(1-\beta_i(\omega)))}{(1-\rho_2)\omega E(C)}.
\end{align*}

We can use the distributional form of Little's law to determine the LST of the marginal queue length distribution of $Q_2$:
\begin{align*}
E\left[z^{N_2}\right] =& \sigma_1(\lambda_2(1-z))\frac{(1-\rho_2)(1-z)\beta_2(\lambda_2(1-z))}{\beta_2(\lambda_2(1-z))-z} \\
&\cdot
\frac{\gamma_1\left(\sum_{i=H,L,2}\lambda_i(1-\beta_i(\lambda_2(1-z)))\right)
- \gamma_1\left(\lambda_2(1-z)+\sum_{i=H,L}\lambda_i(1-\beta_i(\lambda_2(1-z)))\right)}{
\lambda_2(1-\rho_2)(1-z)E(C)}.
\end{align*}

\begin{remark}
The Fuhrmann-Cooper queue length decomposition also holds for all customer classes in a polling system with globally gated service.
\end{remark}

\subsection{Moments}

The expressions for $E(W_H)$ and $E(W_L)$ from Section \ref{momentsgated} also hold in a globally gated polling system, but with a different mean residual cycle time. We only provide the mean waiting time of type 2 customers:
\[E(W_2) = E(S_1) + (1 + 2\rho_H + 2\rho_L + \rho_2) E(C_{1,\textit{res}}).\]

\section{Exhaustive service}\label{exhaustive}

In this section we study the same polling model as in the previous two sections, but the two queues are served exhaustively. The section has the same structure as the other sections, so we start with the derivation of the LST of the joint queue length distribution at polling epochs, followed by the LST of the cycle time distribution. LSTs of the marginal queue length distributions and waiting time distributions are provided in the next subsection. In the last part of the section the mean waiting time of each customer type is studied.

It should be noted that, although we assume that both $Q_1$ and $Q_2$ are served exhaustively, a model in which $Q_2$ is served according to another branching type service discipline, requires only minor adaptations.

\subsection{Joint queue length distribution at polling epochs}

We can derive the joint queue length distribution at the beginning of a cycle for a polling system with two queues and two priority classes in $Q_1$, $P_1(z_H, z_L, z_2)$, directly from  \eqref{p1twoqueues} for $P_1(z_1, z_2)$. Similar to the proof of Lemma \ref{p1_3qs}, we can prove that \[P_1(z_H, z_L, z_2) = P_1\left(\frac{1}{\lambda_1}(\lambda_H z_H+\lambda_L z_L), z_2\right).\]
The same holds for $V_{b_2}(\cdot, \cdot, \cdot)$ and visit completion epochs $V_{c_i}(\cdot, \cdot, \cdot)$, for $i = 1, 2$.

\subsection{Cycle time}

For the cycle time starting with a visit to $Q_1$, \eqref{cycletimelst} is still valid.
However, when studying the waiting time of a specific customer type in an exhaustively served queue, it is convenient to consider the \emph{completion} of a visit to $Q_1$ as the start of a cycle. Hence, in this section the notation $C^*_1$, or the LST of its distribution, $\gamma^*_1(\cdot)$, refers to the cycle time starting at the completion of a visit to $Q_1$. Equation \eqref{cycletimlstVc} gives the LST of the distribution of $C^*_1$.

Using the fact that customers in $Q_1$ are served exhaustively, we can find an alternative, compact expression for $\gamma_1^*(\cdot)$. The type 1 (i.e. both type $H$ and $L$ customers) customers at the beginning of a visit to $Q_1$ are exactly those type 1 customers that have arrived during the previous intervisit time: $P_1(z, 1) = \widetilde{I}_1(\lambda_1(1-z))$. Hence, by setting $\omega=\lambda_1(1-z)$, we get $\widetilde{I}_1(\omega) = P_1(1-\frac{\omega}{\lambda_1}, 1)$, and thus by \eqref{lstCexhaustiveI},
\begin{equation}
\gamma_1^*(\omega) = 
P_1(\pi_1(\omega)-\frac{\omega}{\lambda_1}, 1).\label{lstCexhaustiveShort}
\end{equation}

\subsection{Marginal queue lengths and waiting times}


Analysis of the model with exhaustive service requires a different approach. The key observation, made by \cite{fuhrmanncooper85}, is that a nonpriority polling system from the viewpoint of a type $i$ customer is an $M/G/1$ queue with multiple server vacations. This implies that the Fuhrmann-Cooper decomposition can be used, even though the intervisit times are strongly dependent on the visit times. The $M/G/1$ queue with priorities and vacations can be analysed by modelling the system as a special version of the \emph{nonpriority} $M/G/1$ queue with multiple server vacations, and then applying the results from Fuhrmann and Cooper. This approach has been used by \cite{kellayechiali88} who used the concept of \emph{delay cycles}, and also by \cite{Shanthikumar89} who used \emph{level crossing analysis}; see also \cite{takagi91}. We apply Kella and Yechiali's approach to the polling model under consideration to find the waiting time LST for type $H$ and $L$ customers. In \cite{kellayechiali88} systems with single and multiple vacations, preemptive resume and nonpreemptive service are considered. In the present paper we do not consider preemptive resume, so we only use results from the case labelled as NPMV (nonpreemptive, multiple vacations) in \cite{kellayechiali88}. We consider the system from the viewpoint of a type $H$ and type $L$ customer separately to derive $E[\ee^{-\omega W_H}]$ and $E[\ee^{-\omega W_L}]$.

From the viewpoint of a type $H$ customer and as far as waiting times are concerned, a polling system is a \emph{nonpriority} single server system with multiple vacations. The vacation can either be the intervisit period $I_1$, or the service of a type $L$ customer. The LSTs of these two types of vacations are:
\begin{align}
E[\ee^{-\omega I_1}] &= P_1(1-\omega/\lambda_1, 1), \label{intervisitexhaustive}\\
E[\ee^{-\omega B_L}] &= \beta_L(\omega). \nonumber
\end{align}
Equation \eqref{intervisitexhaustive} follows immediately from the fact that the number of type 1 (i.e. both H and L) customers at the beginning of a visit to $Q_1$ is the number of type 1 customers that have arrived during the previous intervisit period: $P_1(z, 1) = E[\ee^{-(\lambda_1(1-z)) I_1}]$.

We now use the concept of delay cycles, introduced in \cite{kellayechiali88}, to find the waiting time LST of a type $H$ customer. The key observation is that an arrival of a tagged type $H$ customer will always take place within either an $I_H$ cycle, or an $L_H$ cycle. An $I_H$ cycle is a cycle that starts with an intervisit period for $Q_1$, followed by the service of all type $H$ customers that have arrived during the intervisit period, and ends at the moment that no type $H$ customers are left in the system. Notice that at the start of the intervisit period, no type $H$ customers were present in the system either. An $L_H$ cycle is a similar cycle, but starts with the service of a type $L$ customer. This cycle also ends at the moment that no type $H$ customers are left in the system.

The fraction of time that the system is in an $L_H$ cycle is $\frac{\rho_L}{1-\rho_H}$, because type $L$ customers arrive with intensity $\lambda_L$. Each of these customers will start an $L_H$ cycle and the length of an $L_H$ cycle equals $\frac{E(B_L)}{1-\rho_H}$:
\begin{align*}
E(L_H\textrm{ cycle}) &= E(B_L) + \lambda_H E(B_L) E(\textit{BP}_H) \\
&= E(B_L) + \lambda_H E(B_L) \frac{E(B_H)}{1-\rho_H} \\
&= (1+\frac{\rho_H}{1-\rho_H})E(B_L) = \frac{E(B_L)}{1-\rho_H},
\end{align*}
where $E(\textit{BP}_H)$ is the mean length of a busy period of type $H$ customers.

The fraction of time that the system is in an $I_H$ cycle, is $1-\frac{\rho_L}{1-\rho_H} = \frac{1-\rho_1}{1-\rho_H}$. This result can also be obtained by using the argument that the fraction of time that the system is in an intervisit period is the fraction of time that the server is not serving $Q_1$, which is equal to $1-\rho_1$. A cycle which starts with such an intervisit period and stops when all type $H$ customers that arrived during the intervisit period and their type $H$ descendants have been served, has mean length $E(I_1) + \lambda_H E(I_1) E(\textit{BP}_H) = \frac{E(I_1)}{1-\rho_H}$. This also leads to the conclusion that  $\frac{1-\rho_1}{1-\rho_H}$ is the fraction of time that the system is in an $I_H$ cycle. A customer arriving during an $I_H$ cycle views the system as a nonpriority $M/G/1$ queue with multiple server vacations $I_1$; a customer arriving during an $L_H$ cycle views the system as a nonpriority $M/G/1$ queue with multiple server vacations $B_L$.

\cite{fuhrmanncooper85} showed that the waiting time of a customer in an $M/G/1$ queue with server vacations is the sum of two independent quantities: the waiting time of a customer in a corresponding $M/G/1$ queue without vacations, and the residual vacation time. Hence, the LST of the waiting time distribution of a type $H$ customer is:
\begin{equation}
E[\ee^{-\omega W_H}] = \frac{(1-\rho_H)\omega}{\omega-\lambda_H(1-\beta_H(\omega))}\cdot
\left[\frac{1-\rho_1}{1-\rho_H} \cdot \frac{1-\widetilde{I}_1(\omega)}{\omega E(I_1)}+ \frac{\rho_L}{1-\rho_H}\cdot\frac{1-\beta_L(\omega)}{\omega E(B_L)}\right].\label{lstwhexhaustive}
\end{equation}
Equation \eqref{lstwhexhaustive} is in accordance with the more general equation in Section 4.1 in \cite{kellayechiali88}.

\begin{remark}
The LST of the distribution of the waiting time of a high priority customer in a two priority $M/G/1$ queue without vacations is
\begin{equation}
E[\ee^{-\omega W_{H|M/G/1}}] =\frac{(1-\rho_1)\omega+\lambda_L(1-\beta_L(\omega))}{\omega-\lambda_H(1-\beta_H(\omega))}\label{lstwhmg1},\\
\end{equation}
see, e.g., Equation (3.85) in \cite{cohen82}, Chapter III.3. Equation \eqref{lstwhmg1} can be rewritten to \eqref{lstwhexhaustive}, with $\frac{1-\widetilde{I}_1(\omega)}{\omega E(I_1)}$ replaced by 1.
Hence, the waiting time distribution of a high priority customer in a two priority $M/G/1$ queue equals the waiting time distribution of a customer in a nonpriority $M/G/1$ queue with only type $H$ customers, where the server goes on a vacation $B_L$ with probability $\frac{\rho_L}{1-\rho_H}$.
\end{remark}
\begin{remark}
Substitution of \eqref{intervisitC} in \eqref{lstwhexhaustive} expresses $E[\ee^{-\omega W_H}]$ in terms of the LST of the cycle time distribution starting at a visit \emph{completion} to $Q_1$, $\gamma_1^*(\cdot)$:
\begin{equation}
E[\ee^{-\omega W_H}] = \frac{1-\gamma_1^*(\omega - \lambda_H(1-\beta_H(\omega)) - \lambda_L(1-\beta_L(\omega)))
+\lambda_L (1-\beta_L(\omega))E(C)}{(\omega-\lambda_H(1-\beta_H(\omega)))E(C)}.\label{lstwhexhaustiveC}
\end{equation}
\end{remark}
The concept of cycles is not really needed to model the system from the perspective of a type $L$ customer, because for a type $L$ customer the system merely consists of $I_{HL}$ cycles. An $I_{HL}$ cycle is the same as an $I_H$ cycle, discussed in the previous paragraphs, except that it ends when no type $H$ \emph{or $L$} customers are left in the system. So the system can be modelled as a nonpriority $M/G/1$ queue with server vacations. The vacation is the intervisit time $I_1$, plus the service times of all type $H$ customers that have arrived during that intervisit time and their type $H$ descendants. We will denote this extended intervisit time by $I_1^*$ with LST
\[\widetilde{I}_1^*(\omega) = \widetilde{I}_1(\omega+\lambda_H(1-\pi_H(\omega))).\]
The mean length of $I_1^*$ equals $E(I_1^*) = \frac{E(I_1)}{1-\rho_H}$.

We also have to take into account that a busy period of type $L$ customers might be interrupted by the arrival of type $H$ customers. Therefore the alternative system that we are considering will not contain regular type $L$ customers, but customers still arriving with arrival rate $\lambda_L$, whose service time equals the service time of a type $L$ customer in the original model, plus the service times of all type $H$ customers that arrive during this service time, and all of their type $H$ descendants. The LST of the distribution of this extended service time $B_L^*$ is
\[\beta_L^*(\omega) = \beta_L(\omega + \lambda_H(1-\pi_H(\omega))).\]
This extended service time is often called \emph{completion time} in the literature. In this alternative system, the mean service time of these customers equals $E(B_L^*) = \frac{E(B_L)}{1-\rho_H}$. The fraction of time that the system is serving these customers is $\rho_L^* = \frac{\rho_L}{1-\rho_H} = 1 - \frac{1-\rho_1}{1-\rho_H}$.

Now we use the results from the $M/G/1$ queue with server vacations
(starting with the Fuhrmann-Cooper decomposition) to determine the LST of the waiting time distribution for type $L$ customers:
\begin{align}
E[\ee^{-\omega W_L}] =& \frac{(1-\rho_L^*)\omega}{\omega-\lambda_L(1-\beta_L^*(\omega))} \cdot
\frac{1-\widetilde{I}_1^*(\omega)}{\omega E(I_1^*)} \nonumber\\
=&
\frac{(1-\rho_1)(\omega+\lambda_H(1-\pi_H(\omega)))}{\omega-\lambda_L(1-\beta_L(\omega+\lambda_H(1-\pi_H(\omega))))} \cdot
\frac{1-\widetilde{I}_1(\omega+\lambda_H(1-\pi_H(\omega)))}{(\omega + \lambda_H(1-\pi_H(\omega)))E(I_1)}.
\label{lstwlexhaustive}
\end{align}
The last term of \eqref{lstwlexhaustive} is the LST of the distribution of the residual intervisit time, plus the time that it takes to serve all type $H$ customers and their type $H$ descendants that arrive during this residual intervisit time. The first term of \eqref{lstwlexhaustive} is the LST of the waiting time distribution of a low-priority customer in an $M/G/1$ queue with two priorities, without vacations (see e.g. (3.76) in \cite{cohen82}, Chapter III.3).
\begin{remark}
The $M/G/1$ queue with two priorities can be viewed as a nonpriority $M/G/1$ queue with vacations, if we consider the waiting time of type $L$ customers. We only need to rewrite the first term of \eqref{lstwlexhaustive}:
\begin{align*}
E[\ee^{-\omega W_{L|M/G/1}}] =& \frac{(1-\rho_1)(\omega+\lambda_H(1-\pi_H(\omega)))}{\omega-\lambda_L(1-\beta_L(\omega+\lambda_H(1-\pi_H(\omega))))}\\
=&\frac{(1-\rho_L^*)\omega}{\omega-\lambda_L(1-\beta_L^*(\omega))} \cdot \frac{1-\rho_1}{1-\rho_L^*}
\cdot \frac{\omega+\lambda_H(1-\pi_H(\omega))}{\omega} \\
=& E[e^{-\omega W_{L|M/G/1}^*}]
\cdot\left[(1-\rho_H) + \rho_H \frac{1-\pi_H(\omega)}{\omega E(\textit{BP}_H)}\right],
\end{align*}
where $E[e^{-\omega W_{L|M/G/1}^*}]$ is the LST of the waiting time distribution of a customer in an $M/G/1$ queue where customers arrive at intensity $\lambda_L$ and have service requirement LST $\beta_L(\omega + \lambda_H(1-\pi_H(\omega)))$. So with probability $1-\rho_H$ the waiting time of a customer is the waiting time in an $M/G/1$ queue with no vacations, and with probability $\rho_H$ the waiting time of a customer is the sum of the waiting time in an $M/G/1$ queue and the residual length of a vacation, which is a busy period of type $H$ customers.
\end{remark}
\begin{remark}
Substitution of \eqref{intervisitC} in \eqref{lstwlexhaustive} leads to a different expression for $E[\ee^{-\omega W_L}]$:
\begin{align}
E[\ee^{-\omega W_L}] 
&=\frac{1-\gamma^*_1(\omega - \lambda_L(1-\beta_L(\omega+\lambda_H(1-\pi_H(\omega)))))}{(\omega-\lambda_L(1-\beta_L(\omega+\lambda_H(1-\pi_H(\omega)))))E(C)}\nonumber\\
&=E[\ee^{-(\omega - \lambda_L(1-\beta_L(\omega+\lambda_H(1-\pi_H(\omega)))))C^*_{1,\textit{res}}}].\label{lstwlexhaustiveC}
\end{align}
\end{remark}
The waiting time of type 2 customers is not affected at all by the fact that $Q_1$ contains multiple classes of customers, so \eqref{lstwexhaustive} is still valid for $E(\ee^{-\omega W_2})$.

We will refrain from mentioning the PGFs of the marginal queue length distributions here, because they can be obtained by applying the distributional form of Little's law as we have done before.

\subsection{Moments}

The mean waiting times for high and low priority customers can be found by differentiation of \eqref{lstwhexhaustive} and \eqref{lstwlexhaustive}:
\begin{align*}
E(W_H) &= \frac{\rho_H E(B_{H,\textit{res}})+\rho_L E(B_{L,\textit{res}})}{1-\rho_H}+\frac{1-\rho_1}{1-\rho_H}E(I_{1,\textit{res}}),\\
E(W_L) &= \frac{\rho_H E(B_{H,\textit{res}}) + \rho_L E(B_{L,\textit{res}})}{(1-\rho_H)(1-\rho_1)} + \frac{1}{1-\rho_H}E(I_{1,\textit{res}}).
\end{align*}
Differentiation of \eqref{lstwhexhaustiveC} and \eqref{lstwlexhaustiveC} leads to alternative expressions, that can also be found in \cite{wierman07}.
\begin{align*}
E(W_H) &= \frac{(1-\rho_1)^2}{1-\rho_H}\frac{E({C^*_1}^2)}{2E(C)},\\
E(W_L) &= \frac{(1-\rho_1)^2}{(1-\rho_H)(1-\rho_1)}\frac{E({C^*_1}^2)}{2E(C)}\\
&= \left(1-\frac{\rho_L}{1-\rho_H}\right)\frac{E({C^*_1}^2)}{2E(C)}.
\end{align*}

\section{Example}\label{numericalexample}

Consider a polling system with two queues, and assume exponential service times and switch-over times. Suppose that $\lambda_1 = \frac{6}{10}, \lambda_2 = \frac{2}{10}, E(B_1) = E(B_2) = 1, E(S_1) = E(S_2) = 1$. The workload of this polling system is $\rho = \frac{8}{10}$. This example is extensively discussed in \cite{winands06} where MVA was used to compute mean waiting times and mean residual cycle times for the gated and exhaustive service disciplines.

In this example we show that the performance of this system can be improved by giving higher priority to jobs with smaller service times. We define a threshold $t$ and divide the jobs into two classes: jobs with a service time less than $t$ receive high priority, the other jobs receive low priority.
In Figures \ref{fig:gated} and \ref{fig:exhaustive} the mean waiting times of customers in $Q_1$ are shown as a function of the threshold $t$. The following four cases are distinguished:
\begin{itemize}
\item the mean waiting time of the low priority customers in $Q_1$ (indicated as ``Type L'');
\item the mean waiting time of the high priority customers in $Q_1$ (indicated as ``Type H'');
\item a weighted average of the above two mean waiting times: $\frac{\lambda_L}{\lambda_1}E(W_L) + \frac{\lambda_H}{\lambda_1}E(W_H)$ (indicated as ``Type 1 with priorities''). This can be interpreted as the mean waiting time of an arbitrary customer in $Q_1$;
\item the mean waiting time of an arbitrary customer in $Q_1$ if no priority rules would be applied to this queue (indicated as ``Type 1 no priorities''). In this situation there is no such thing as high and low priority customers, so the mean waiting time does not depend on $t$, and has already been computed in \cite{winands06}.
\end{itemize}
The figures show that a unique optimal threshold exists that minimises the mean weighted waiting time for customers in $Q_1$. This value depends on the service discipline used and is discussed in \cite{wierman07}. In this example the optimal threshold is 1 for gated, and 1.38 for exhaustive. Figure \ref{fig:gated} confirms that the mean waiting times for type $H$ and $L$ customers in the gated model only differ by a constant value: $E(W_L) - E(W_H) = \rho_1 E(C_{1,\textit{res}})$. For globally gated service no figure is included, because we again have $E(W_L) - E(W_H) = \rho_1 E(C_{1,\textit{res}})$.
The mean residual cycle time is different from the one in the gated model, but this does not affect the optimal threshold which is still $t=1$.

In the exhaustive model we have the following relation:
\[E(W_L) - E(W_H) = \frac{\rho_1(1-\rho_1)}{1-\rho_H}E(C^*_{1,\textit{res}}).\]
If we increase threshold $t$, the fraction of customers in $Q_1$ that receive high priority grows, and so does their mean service time. This means that $\rho_H$ increases as $t$ increases, so  $E(W_L) - E(W_H)$ gets bigger, which can be seen in Figure \ref{fig:exhaustive}. Notice that $\frac{E(W_H)}{E(W_L)} = 1-\rho_1$, so it does not depend on $t$.

\begin{figure}[h!]
\begin{center}
\includegraphics[width=\linewidth]{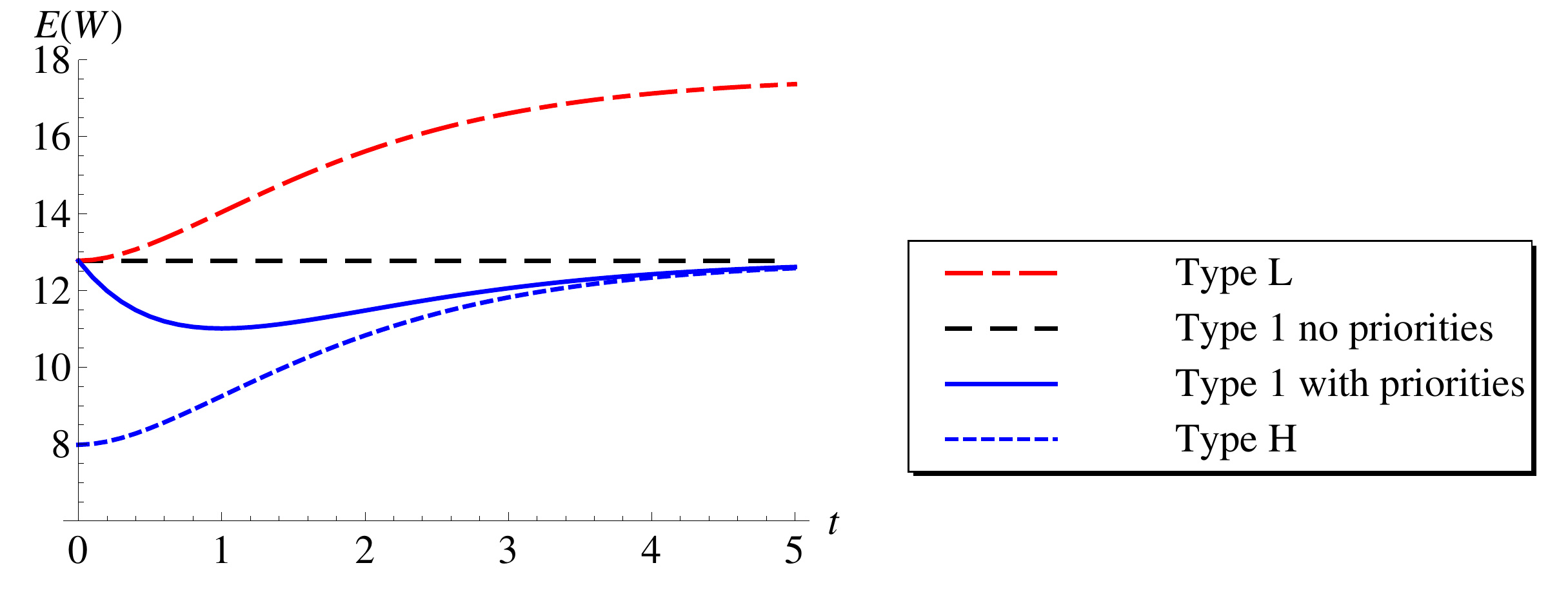}
\end{center}
\caption{Mean waiting time of customers in $Q_1$ in the gated polling system, versus threshold $t$.\label{fig:gated}}
\end{figure}

\begin{figure}[h!]
\begin{center}
\includegraphics[width=\linewidth]{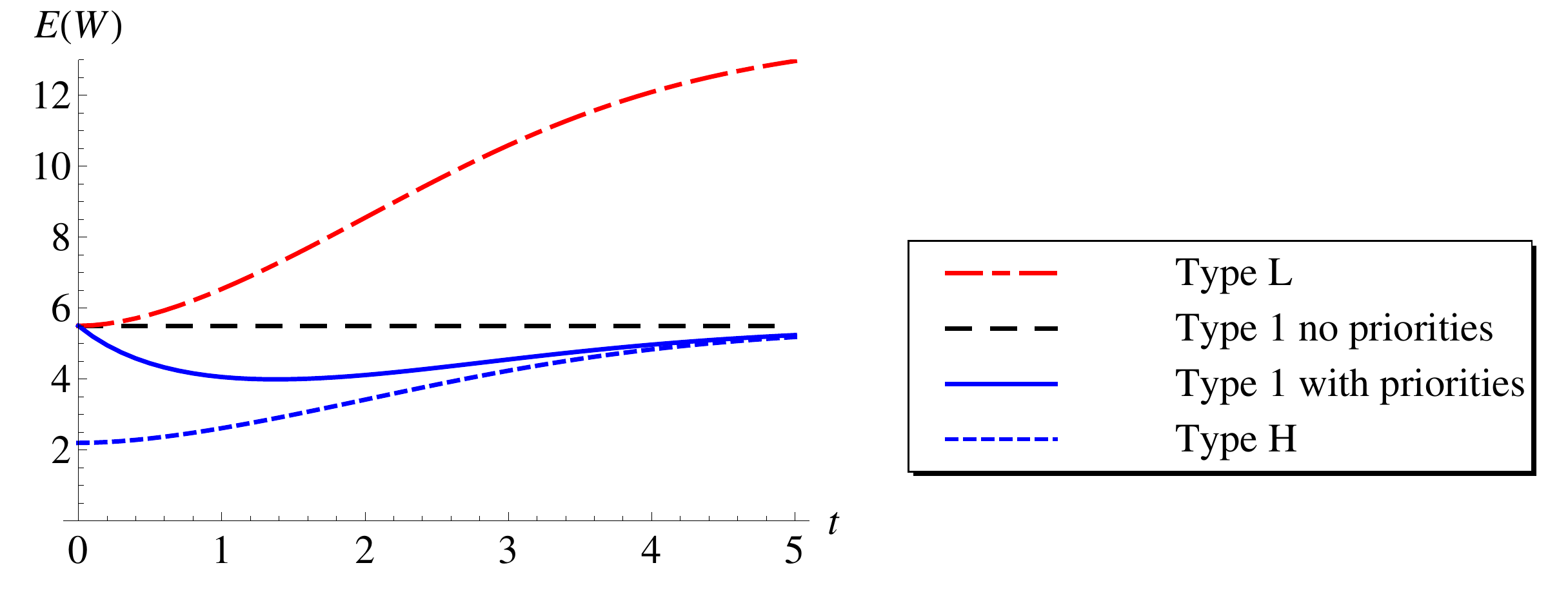}
\end{center}
\caption{Mean waiting time of customers in $Q_1$ in the exhaustive polling system, versus threshold $t$.\label{fig:exhaustive}}
\end{figure}

It is interesting to also consider the variance, or rather the standard deviation of the waiting time. Figures \ref{fig:gatedvar} and \ref{fig:exhaustivevar} show the standard deviation of the type $H$ and $L$ customers versus the threshold $t$. The figures also show the standard deviation of an arbitrary customer in $Q_1$, with and without priorities. The figures indicate that the waiting times in the gated system have smaller standard deviations than in the exhaustive case. In this example, the introduction of priorities affects the standard deviation of an arbitrary type 1 customer only slightly. However, it is interesting to zoom in to investigate the influence of threshold $t$. Figure \ref{fig:gatedexhaustivevarzoom} contains zoomed versions of Figures \ref{fig:gatedvar} and \ref{fig:exhaustivevar} and indicates that the threshold $t$ that minimises the overall mean waiting time of type $1$ customers in the priority system does not minimise the standard deviation. In fact, changing threshold $t$ affects the entire service time distributions $B_H$ and $B_L$, which results in two local minima for the standard deviation as function of threshold $t$.

\begin{figure}[t!]
\begin{center}
\includegraphics[width=\linewidth]{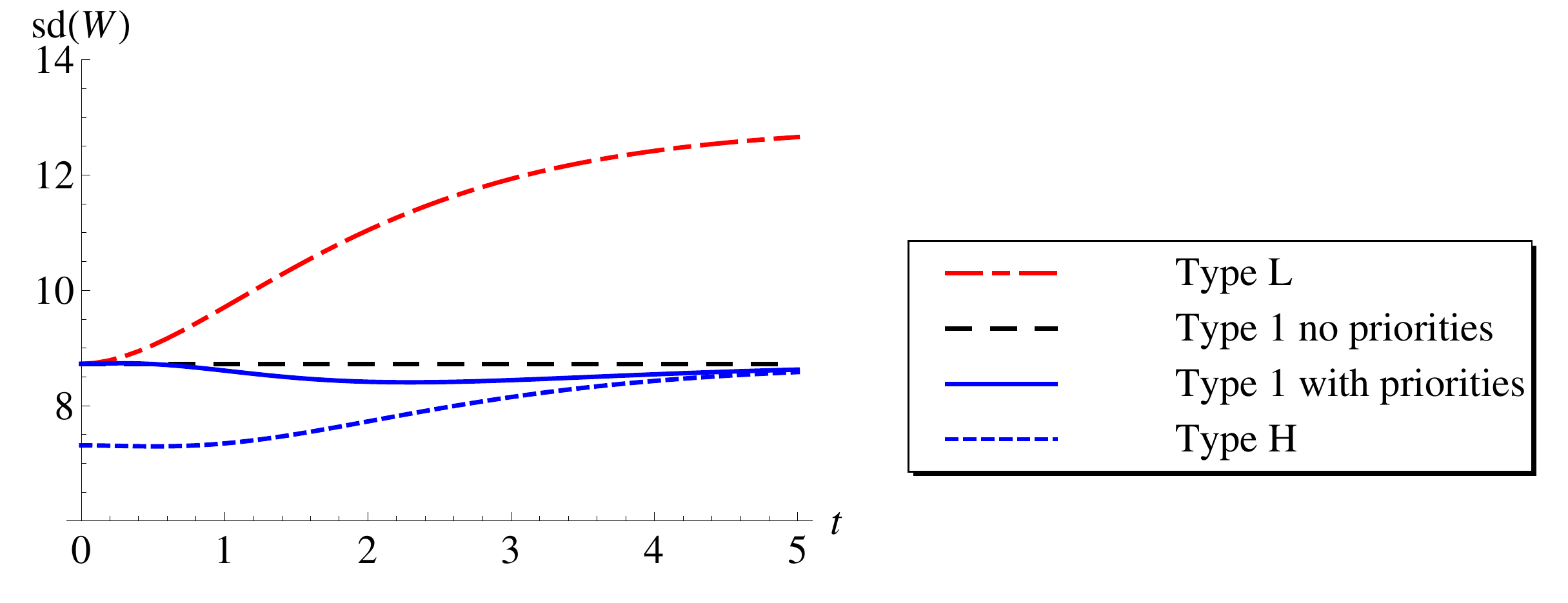}
\end{center}
\caption{Standard deviation of the waiting time of customers in $Q_1$ in the gated polling system, versus threshold $t$.\label{fig:gatedvar}}
\end{figure}

\begin{figure}[h!]
\begin{center}
\includegraphics[width=\linewidth]{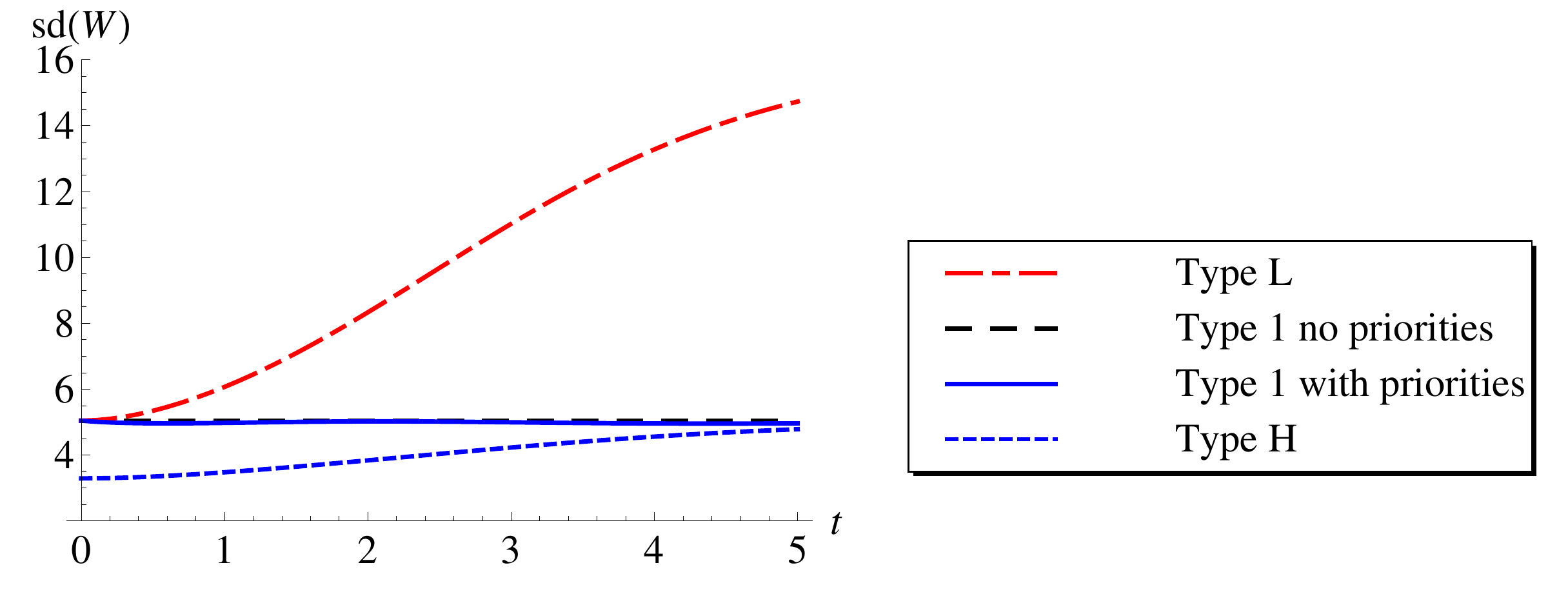}
\end{center}
\caption{Standard deviation of the waiting time of customers in $Q_1$ in the exhaustive polling system, versus threshold $t$.\label{fig:exhaustivevar}}
\end{figure}

\begin{figure}[t!]
\begin{center}
\includegraphics[width=0.45\linewidth]{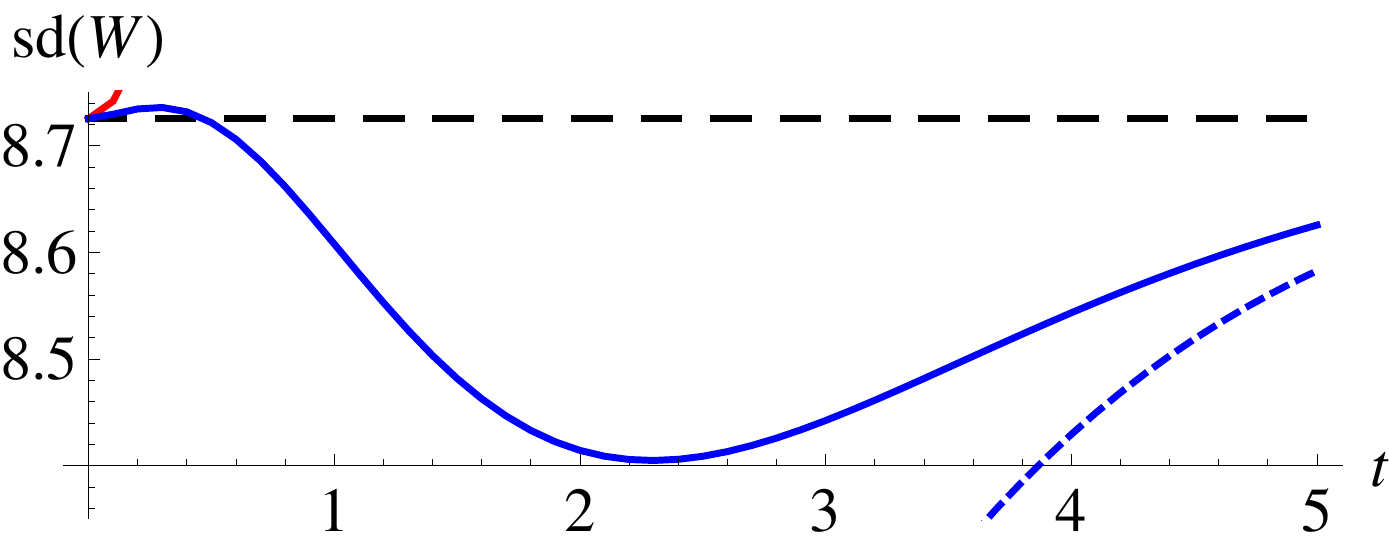}
\hfill
\includegraphics[width=0.45\linewidth]{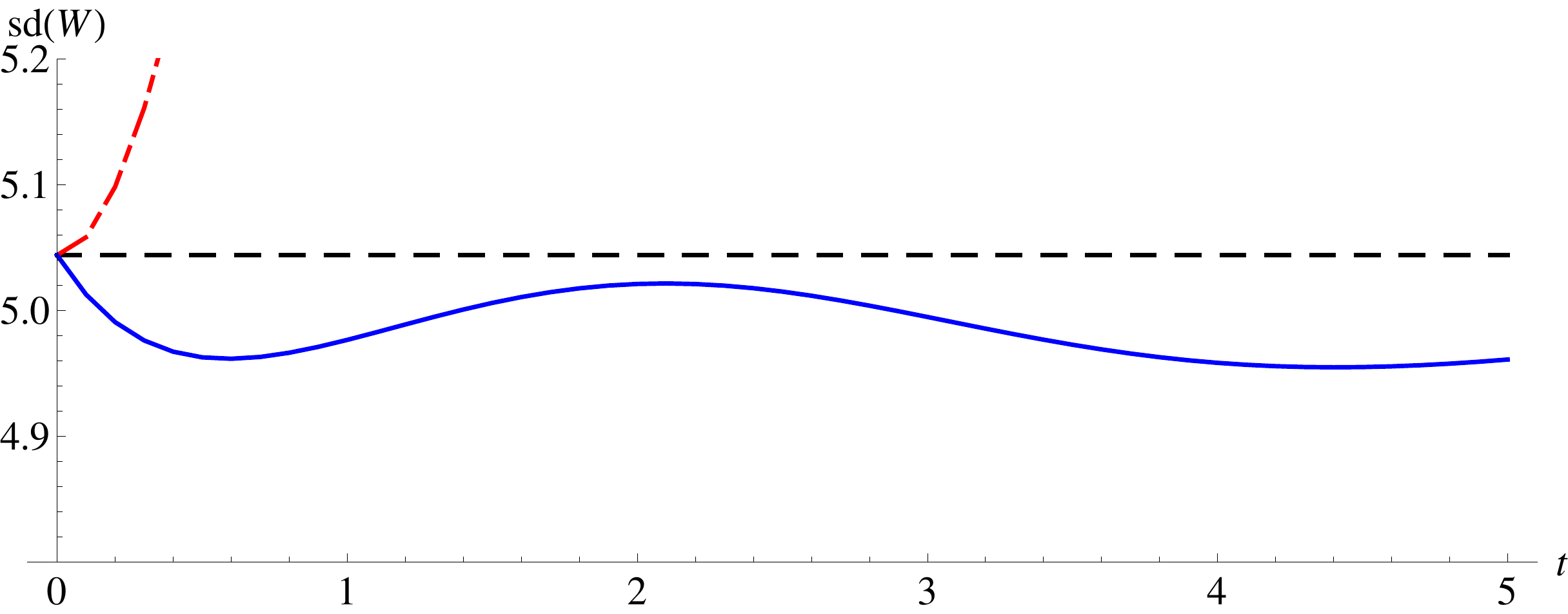}
\end{center}
\caption{Zoomed versions of Figures \ref{fig:gatedvar} (left) and \ref{fig:exhaustivevar} (right).\label{fig:gatedexhaustivevarzoom}}
\end{figure}

\section{Possible extensions and future research}\label{extensions}

The polling system studied in the present paper leaves many possibilities for extensions or variations. In this section we discuss some of them.

\paragraph{Multiple queues and priority levels.} Probably the most obvious extension of the model under consideration, is a polling system with any number of queues and any number of priority levels in each queue. In recent research \citep{boonadanboxma2008}, we have discovered that such a polling model can be analysed in detail. Each queue can have its own service discipline, either exhaustive or (synchronised) gated.

\paragraph{Preemptive resume.} In the present paper, the service of low priority customers is not interrupted by the arrival of a high priority customer. If we allow for service interruptions, these would only take place in a queue with exhaustive service, since (globally) gated service forces high priority customers to wait behind the gate. We note that allowing service interruptions does not affect the joint queue length distributions at polling instants, nor the cycle time.
Also the waiting time of low priority customers is unaffected (but they might have a longer \emph{sojourn time}). It only affects the waiting time of high priority customers, because they do not have to wait for a residual service time of a low priority customer. The LST of the waiting time distribution of a high priority customer if service is preemptive resume, is:
\begin{equation*}
E[\ee^{-\omega W_H}] = \frac{(1-\rho_H)\omega}{\omega-\lambda_H(1-\beta_H(\omega))}\cdot
\left[\frac{1-\rho_1}{1-\rho_H} \cdot \frac{1-\widetilde{I}_1(\omega)}{\omega E(I_1)}+ \frac{\rho_L}{1-\rho_H}\right].
\end{equation*}

\paragraph{Mixed gated/exhaustive service.} In the present paper, customers in $Q_1$ receive either exhaustive or (globally) gated service. One may consider serving each priority level according to a different service discipline. In \cite{boonadan2008}, high priority customers receive exhaustive service, whereas low priority customers receive gated service. This gives high priority customers an additional advantage, but it turns out that for low priority customers this strategy may be better than, e.g., gated service for all priority levels. A mixture of globally gated service for low priority customers and exhaustive service for high priority customers can be analysed similarly.

The ``opposite'' strategy, where low priority customers are served exhaustively and high priority customers are served according to the gated service discipline is easier to analyse, since we can model it as a nonpriority polling model with $Q_1$ replaced by two queues, $Q_{H}$ and $Q_L$, containing the type $H$ and type $L$ customers and having gated and exhaustive service respectively.

\paragraph{Partially gated.} A variant of the gated service discipline is partially gated service: every customer, type $H$ or $L$, standing in front of the gate is served during a visit with a fixed probability $p$, and is not served with probability $1-p$. The probability $p$ might even depend on the customer type. Whether a rejected customer is eligible for service in the next cycle, or leaves the system, does not matter. Both situations can be analysed.

\paragraph{Different polling sequences.} We assume that the server alternates between $Q_1$ and $Q_2$. A different way of introducing priorities to a polling system is by increasing the frequency of visits to a queue within a cycle. One can, e.g., decide to visit $Q_1$ two consecutive times if gated service is used. Or one can think of a system where the server switches to $Q_j$ after completing a visit to $Q_i$ with probability $p_{ij}$.

\paragraph{Large setup times.} \cite{winandsPhD} establishes fluid limits for polling systems with any branching type service discipline and \emph{deterministic} switch-over times tending to infinity. The scaled waiting time distribution is shown to converge to a uniform distribution with bounds that can be computed explicitly. The results are relevant to applications in production systems, where large setup times are common. These fluid limits can also be computed for the polling model that is discussed in the present paper and give explicit insight in when each of the discussed service disciplines is optimal.

\bibliographystyle{abbrvnat}

\end{document}